\documentclass[preprint,11pt]{elsarticle}
\usepackage[left=2.5cm, right=2.5cm, top=2cm, bottom=2cm]{geometry}

\usepackage{amsmath, amsthm, amssymb}
\usepackage{graphicx}
\usepackage{tikz}
\usepackage{url}
\usetikzlibrary{calc}
\usetikzlibrary{positioning}
\usepackage{booktabs}
\usepackage{enumerate}

\newtheorem{lemma}{Lemma}
\newtheorem{theorem}{Theorem}
\newtheorem{corollary}{Corollary}

\newtheorem{definition}{Definition}

\begin{document}

\begin{frontmatter}

\title{A general algebraic characterization of extremal $\Delta$-trees and $k$-cyclic graphs for degree-homogeneous VDB indices}

\author[nis]{Milan Ba\v{s}i\'{c}\corref{cor1}}
\ead{basic_milan@yahoo.com}
\ead{milan.basic@pmf.edu.rs}

\cortext[cor1]{Corresponding author}

\address[nis]{Faculty of Sciences and Mathematics, University of Ni\v{s}, Vi\v{s}egradska 33, 18000 Ni\v{s}, Serbia}

\begin{abstract}
We present a general algebraic method for maximizing a vertex-degree-based (VDB) topological index $TI(G) = \sum_{uv \in E(G)} f(d(u),d(v))$ over the class of $n$-vertex trees with a prescribed maximum degree $\Delta \ge 3$. By expressing the edge multiplicities of the basis $\mathcal{B} = \{(1,\Delta), (\Delta,\Delta)\}$ via degree-sum identities, we reduce the maximization of $TI(T)$ to the minimization of an explicit discrete penalty function $G(x,y)$ over the non-basic edges. We formulate five analytical conditions on the weight function $f(x,y)$ that dictate the optimal edge distributions, yielding a complete characterization of the extremal trees across all congruence classes of $n \pmod{\Delta - 1}$. Additionally, we provide their explicit construction by attaching leaves to a core of internal vertices. Applying this framework, we obtain the complete characterization of extremal $\Delta$-trees for the Degree-Ratio Sombor ($DRSO$), Symmetric Division Deg ($SDD$), and Hyperbolic Sombor ($HSO$) indices for all $\Delta \ge 3$. Furthermore, the analytical conditions identify the exact degree thresholds for the Arithmetic-Geometric ($AG$) and Diminished Sombor ($DSO$) indices, showing that the same extremal family maximizes $AG$ if and only if $3 \le \Delta \le 8$, and $DSO$ if and only if $3 \le \Delta \le 4$. Finally, since changing the edge count from $m = n - 1$ to $m = n - 1 + k$ affects only the constant term in the decomposition while leaving the penalty function $G(x,y)$ invariant, the method extends directly to connected $k$-cyclic graphs.
\end{abstract}

\begin{keyword}
vertex-degree-based topological indices \sep trees with given maximum degree \sep k-cyclic graphs \sep discrete optimization \sep penalty function method \sep Sombor-type indices

\MSC[2020] 05C09 \sep 05C35 \sep 90C05 \sep 92E10
\end{keyword}

\end{frontmatter}

\section{Introduction}

Let $G=(V,E)$ be a simple connected graph of order $n=|V|$ and size $m=|E|$. For a vertex $u \in V$, its degree is denoted by $d(u)$, and the maximum vertex degree in $G$ is denoted by $\Delta$. A vertex of degree $1$ is called a pendant vertex (or a leaf), and an edge incident to it is a pendant edge. An edge $uv \in E$ connecting vertices of degrees $d(u)=i$ and $d(v)=j$ is referred to as an $(i,j)$-type edge, and its multiplicity (the number of such edges in $G$) is denoted by $m_{i,j}$. Throughout this paper, $T$ denotes a tree of order $n$, while $P_n$ and $S_n$ denote the path and star graphs on $n$ vertices, respectively. When $\Delta \le 4$, a tree $T$ is referred to in chemical graph theory as a \emph{molecular} (or \emph{chemical}) tree. For standard graph-theoretic terminology not defined here, we refer the reader to \cite{bondy2008graph}.

A vertex-degree-based (VDB) topological index (or bond incident degree index) is a graph invariant of the general form
\begin{equation*}
    TI(G) = \sum_{uv \in E} f(d(u), d(v)) = \sum_{1 \le i \le j \le \Delta} f(i,j)\, m_{i,j},
\end{equation*}
where $f(x,y)$ is a symmetric real-valued weight function. Following the introduction of the Sombor index $SO(G) = \sum_{uv \in E} \sqrt{d(u)^2 + d(v)^2}$ by Gutman in 2021 \cite{gutman2021geometric}, geometric and degree-normalized VDB invariants have been widely investigated in mathematical chemistry (see the recent survey \cite{gutman2026survey}). Particular attention has been given to normalized and ratio-based variants introduced over the past two years, including the \emph{Diminished Sombor} ($DSO$) index \cite{movahedi2026diminished, movahedi2025diminished, movahedi2026maximal}, the \emph{Hyperbolic} \emph{Sombor} ($HSO$) index \cite{barman2026geometric}, and the \emph{Degree-Ratio Sombor} ($DRSO$) index \cite{nyauli2027geometric, ali2027degree, gutman2027comparing, gutman2026extremal, shaoa2027degree}. Alongside classical degree-ratio descriptors from the Adriatic family—such as the \emph{Symmetric Division Deg} ($SDD$) index \cite{vukicevic2010bond, ali2020on, ghorbani2021new} and the \emph{Arithmetic-Geometric} ($AG$) index \cite{shegehalli2015arithmetic, vujosevic2021arithmetic, rodriguez2021some, li2022sharp}—these five invariants share a key analytical property: their weight functions $f(x,y)$ are homogeneous of degree zero, and hence constant along the main diagonal $x=y$ (Table~\ref{tab:indices}).

\begin{table}[htbp]
\centering
\caption{The five degree-homogeneous VDB topological indices analyzed in this study, their weight functions $f(x,y)$ (for $1 \le x \le y \le \Delta$), diagonal values $f(x,x)$, and validity domains for Theorem~\ref{thm:alternative_structure}.}
\label{tab:indices}
\renewcommand{\arraystretch}{1.45}
\begin{tabular}{llllc}
\hline
\textbf{Topological Index} & \textbf{Symbol} & \textbf{Weight function $f(x,y)$} & \textbf{$f(x,x)$} & \textbf{Valid domain} \\
\hline
Degree-Ratio Sombor \cite{nyauli2027geometric} & $DRSO$ & $\displaystyle\sqrt{\left(\frac{x}{y}\right)^2 + \left(\frac{y}{x}\right)^2}$ & $\sqrt{2}$ & $\Delta \ge 3$ \\
Symmetric Division Deg \cite{vukicevic2010bond} & $SDD$ & $\displaystyle\frac{x^2+y^2}{xy} = \frac{x}{y} + \frac{y}{x}$ & $2$ & $\Delta \ge 3$ \\
Hyperbolic Sombor \cite{barman2026geometric} & $HSO$ & $\displaystyle\frac{\sqrt{x^2+y^2}}{x}$ \quad ($x \le y$) & $\sqrt{2}$ & $\Delta \ge 3$ \\
Arithmetic-Geometric \cite{shegehalli2015arithmetic} & $AG$ & $\displaystyle\frac{x+y}{2\sqrt{xy}}$ & $1$ & $3 \le \Delta \le 8$ \\
Diminished Sombor \cite{movahedi2026diminished} & $DSO$ & $\displaystyle\frac{\sqrt{x^2+y^2}}{x+y}$ & $\displaystyle\frac{\sqrt{2}}{2}$ & $3 \le \Delta \le 4$ \\
\hline
\end{tabular}
\end{table}

Determining the extremal trees that maximize a VDB index among all $n$-vertex trees with a prescribed maximum degree $\Delta$ is a fundamental problem in extremal graph theory. While over unconstrained $n$-vertex trees the maximum of such indices is typically attained by the star graph $S_n$ (as shown for $DRSO$ in \cite{ali2027degree, gutman2026extremal} and for $HSO$ in \cite{barman2026geometric}), imposing an upper bound $\Delta$ on the vertex degrees partitions the extremal problem into $\Delta - 1$ congruence classes of $n \pmod{\Delta - 1}$. Existing studies generally address this problem in one of two ways:
\begin{enumerate}
    \item \emph{Upper bounds tight for a single congruence class:} Several works derive general upper bounds in terms of $n$, $m$, and $\Delta$ (see, e.g., Proposition~2 in \cite{ali2027degree} for the $DRSO$ index, or \cite{rodriguez2021some} for the $AG$ index), or systematically identify optimal edge bases for various VDB indices via linear programming \cite{basic2026unified}. However, equality in such bounds holds if and only if the tree contains exclusively basic edges—such as $(1,\Delta)$- and $(\Delta,\Delta)$-type edges—which is realizable only in a single congruence class (e.g., $n \equiv 2 \pmod{\Delta - 1}$). The remaining $\Delta - 2$ congruence classes, where internal vertices of intermediate degrees must occur, remain unresolved.
    \item \emph{Index-by-index characterizations:} When complete characterizations across all congruence classes are pursued, they are invariably carried out for one index at a time, most often via combinatorial edge-switching arguments restricted to molecular trees ($\Delta \le 4$). For instance, extremal chemical trees ($\Delta \le 4$) have been characterized for the $SDD$ index \cite{ali2020on}, general chemical BID indices \cite{albalahi2023bond}, and recently for the $DSO$ \cite{movahedi2026maximal, mondal2026extremal} and $HSO$ \cite{hussain2027extremal} indices. Extending such proofs to larger $\Delta$ typically requires separate investigations for each degree—as seen with the forgotten topological index ($F$-index) \cite{furtula2015forgotten, abdo2017extremal} and the $\sigma$-irregularity index, where extremal trees were resolved first for $\Delta \le 4$ \cite{kovijanic2024sigma}, and subsequently for $\Delta = 5$ \cite{dimitrov2026sigma5}, before general characterizations for arbitrary $\Delta$ were completed combinatorially by Knor, Sedlar, and \v{S}krekovski \cite{knor2026sigma} and algebraically in \cite{basic2027sigma}.
\end{enumerate}

The aim of this paper is to unify these two lines of research. By extending the basis decomposition approach of \cite{basic2026unified} to all congruence classes of $n \pmod{\Delta - 1}$, we provide a complete characterization of extremal $\Delta$-trees for a broad class of VDB indices simultaneously. Rather than treating each invariant and each maximum degree $\Delta$ separately, we formulate the problem in the edge-multiplicity space relative to the basis $\mathcal{B} = \{(1,\Delta), (\Delta,\Delta)\}$. Using the edge-sum and degree-sum identities, we eliminate the basic edge multiplicities $m_{1,\Delta}$ and $m_{\Delta,\Delta}$ and reduce the maximization of any VDB index $TI(T)$ to the minimization of an explicit discrete penalty sum $\sum_{(i,j)\notin \mathcal{B}} G(i,j)\,m_{i,j}$. We then establish five analytical conditions, \textbf{(P1)}--\textbf{(P5)}, on the weight function $f(x,y)$ that govern the non-basic penalty hierarchy (Lemma~\ref{lem:universal_bounds}) and uniquely determine the extremal tree family $\mathcal{T}_{\mathrm{opt}}(n,\Delta)$ across all congruence classes of $n \pmod{\Delta - 1}$ (Theorem~\ref{thm:alternative_structure}).

This framework yields two main contributions:
\begin{itemize}
    \item \textbf{Complete characterizations for all $\Delta \ge 3$:} In Corollaries~\ref{cor:drso_all_delta}, \ref{cor:sdd_all_delta}, and \ref{cor:hso_all_delta}, we prove that the weight functions of the $DRSO$, $SDD$, and $HSO$ indices satisfy \textbf{(P1)}--\textbf{(P5)} for all $\Delta \ge 3$. This gives the first complete characterization of extremal $\Delta$-trees for the $DRSO$, $HSO$, and $SDD$ indices (including molecular trees $\Delta \le 4$).
    \item \textbf{Exact degree thresholds and structural transitions:} Our analytical criteria identify the exact threshold of $\Delta$ beyond which the basis $\mathcal{B} = \{(1,\Delta), (\Delta,\Delta)\}$ ceases to be optimal. Specifically, we prove that the $AG$ index satisfies \textbf{(P1)}--\textbf{(P5)} if and only if $3 \le \Delta \le 8$ (Corollary~\ref{cor:ag_small_delta}), whereas the $DSO$ index satisfies them if and only if $3 \le \Delta \le 4$ (Corollary~\ref{cor:dso_small_delta}), recovering the recent characterization of $DSO$-extremal molecular trees \cite{mondal2026extremal} and explaining why the extremal topology changes for $\Delta \ge 5$.
\end{itemize}

Finally, a key structural feature of our algebraic decomposition is that the penalty function $G(x,y)$ in \eqref{eq:algebraic_G} is independent of the total number of edges $m$. Specifically, replacing trees ($m = n - 1$) with general connected graphs of prescribed order $n$ and size $m = n - 1 + k$ (i.e., connected $k$-cyclic graphs, where $k \ge 0$ is the cyclomatic number) alters only the right-hand side of the edge-sum identity \eqref{eq:edge_sum}. This shifts the constant term $\mathcal{C}(n,\Delta)$ while leaving the penalty coefficients $G(x,y)$ unchanged. Consequently, without loss of algebraic generality, we carry out the derivation and constructive tree realizations for $k = 0$ in Section~2, and present the direct extension to connected $k$-cyclic graphs ($k \ge 1$) in Section~3.

\section{Algebraic Decomposition Relative to the Basis $\mathcal{B} = \{(1,\Delta), (\Delta,\Delta)\}$}

To analyze extremal trees with respect to the edge-multiplicity basis $\mathcal{B} = \{(1, \Delta), (\Delta, \Delta)\}$, we express the basic edge counts directly in terms of the non-basic variables.

For any tree $T$ of order $n$, the edge multiplicities $m_{i,j}$ satisfy two identities representing the total number of edges and the total number of vertices:
\begin{align}
    \sum_{1 \le i \le j \le \Delta} m_{i,j} &= n - 1, \label{eq:edge_sum} \\
    \sum_{1 \le i \le j \le \Delta} \left( \frac{1}{i} + \frac{1}{j} \right) m_{i,j} &= n. \label{eq:vertex_sum}
\end{align}

To isolate the basic variables $m_{1,\Delta}$ and $m_{\Delta,\Delta}$, we separate them from the non-basic summations. From \eqref{eq:edge_sum}, we have
\begin{equation}
\label{eq:m_delta_delta_sub}
    m_{\Delta,\Delta} = n - 1 - m_{1,\Delta} - \sum_{(i,j) \notin \mathcal{B}} m_{i,j}.
\end{equation}
Substituting this expression into \eqref{eq:vertex_sum} yields
\begin{equation*}
    \left( 1 + \frac{1}{\Delta} \right) m_{1,\Delta} + \frac{2}{\Delta} \left( n - 1 - m_{1,\Delta} - \sum_{(i,j) \notin \mathcal{B}} m_{i,j} \right) + \sum_{(i,j) \notin \mathcal{B}} \left( \frac{1}{i} + \frac{1}{j} \right) m_{i,j} = n.
\end{equation*}
Grouping the terms involving $m_{1,\Delta}$ on the left-hand side gives
\begin{equation*}
    \left( 1 - \frac{1}{\Delta} \right) m_{1,\Delta} = n - \frac{2(n-1)}{\Delta} - \sum_{(i,j) \notin \mathcal{B}} \left( \frac{1}{i} + \frac{1}{j} - \frac{2}{\Delta} \right) m_{i,j}.
\end{equation*}
Multiplying by $\frac{\Delta}{\Delta - 1}$ gives $m_{1,\Delta}$, and substituting back into \eqref{eq:m_delta_delta_sub} gives $m_{\Delta,\Delta}$:
\begin{align}
    m_{1,\Delta} &= \frac{(\Delta - 2)n + 2}{\Delta - 1} - \frac{\Delta}{\Delta - 1} \sum_{(i,j) \notin \mathcal{B}} \left( \frac{1}{i} + \frac{1}{j} - \frac{2}{\Delta} \right) m_{i,j}, \label{eq:m_1_delta_final} \\
    m_{\Delta,\Delta} &= \frac{n - \Delta - 1}{\Delta - 1} + \sum_{(i,j) \notin \mathcal{B}} \left[ \frac{\Delta}{\Delta - 1} \left( \frac{1}{i} + \frac{1}{j} - \frac{2}{\Delta} \right) - 1 \right] m_{i,j}. \label{eq:m_delta_delta_final}
\end{align}

Writing $TI(T) = \sum f(i,j) m_{i,j}$ in terms of basic and non-basic components, we have
\begin{equation}
\label{eq:ti_expanded}
    TI(T) = f(1,\Delta)m_{1,\Delta} + f(\Delta,\Delta)m_{\Delta,\Delta} + \sum_{(i,j) \notin \mathcal{B}} f(i,j) m_{i,j}.
\end{equation}
Substituting \eqref{eq:m_1_delta_final} and \eqref{eq:m_delta_delta_final} into \eqref{eq:ti_expanded} and collecting the coefficients of $m_{i,j}$ yields
\begin{equation}
    TI(T) = \mathcal{C}(n, \Delta) - \sum_{(i,j) \notin \mathcal{B}} \left[ f(\Delta,\Delta) - f(i,j) + \frac{\Delta}{\Delta - 1} \big( f(1,\Delta) - f(\Delta,\Delta) \big) \left( \frac{1}{i} + \frac{1}{j} - \frac{2}{\Delta} \right) \right] m_{i,j},
\end{equation}
where $\mathcal{C}(n, \Delta) = f(1,\Delta)\frac{(\Delta - 2)n + 2}{\Delta - 1} + f(\Delta,\Delta)\frac{n - \Delta - 1}{\Delta - 1}$ depends only on $n$ and $\Delta$.

Setting $C_\Delta = \frac{\Delta}{\Delta - 1} \big( f(1,\Delta) - f(\Delta,\Delta) \big)$, we obtain the penalty representation
\begin{equation}
\label{eq:algebraic_G}
    TI(T) = \mathcal{C}(n, \Delta) - \sum_{(i,j) \notin \mathcal{B}} G(i,j) m_{i,j}, \quad \text{where} \quad G(x,y) = f(\Delta,\Delta) - f(x,y) + C_\Delta \left( \frac{1}{x} + \frac{1}{y} - \frac{2}{\Delta} \right).
\end{equation}
Thus, maximizing $TI(T)$ reduces to minimizing the non-basic penalty sum $\sum_{(i,j)\notin \mathcal{B}} G(i,j) m_{i,j}$.

\subsection{Properties of the Weight Function}

Let $f(x,y)$ be a continuous symmetric weight function satisfying the following properties for a given maximum degree $\Delta$:
\begin{itemize}
    \item \textbf{(P1)} $f(x,x) = f(\Delta,\Delta)$ for all $x \le \Delta$.
    \item \textbf{(P2)} $f(x,\Delta) > f(\Delta,\Delta)$ for all $x < \Delta$.
    \item \textbf{(P3)} $\frac{\partial f(x,y)}{\partial y} > 0$ for all $x < y$.
    \item \textbf{(P4)} $\frac{C_\Delta}{t^2} > \left| \frac{\partial f(t,x)}{\partial t} \right|$ for all $1 \le t < x < \Delta$, where $C_\Delta = \frac{\Delta}{\Delta - 1} \big( f(1,\Delta) - f(\Delta,\Delta) \big)$.
    \item \textbf{(P5)} $f(\Delta,\Delta) - f(x,\Delta) + C_\Delta \left( \frac{1}{x} - \frac{1}{\Delta} \right) > 0$ for all $2 \le x \le \Delta - 1$.
\end{itemize}

These five analytical conditions directly govern the behavior and ordering of the penalty function $G(x,y)$ over all non-basic edge types, as established in the following lemma.

\begin{lemma}[Penalty Bounds]
\label{lem:universal_bounds}
If $f(x,y)$ satisfies properties (P1)–(P5), then $G(x,y)$ satisfies:
\begin{enumerate}[(i)]
    \item $C_\Delta > 0$.
    \item $G(x,j) > G(x,\Delta)$ for all $x < j < \Delta$.
    \item $G(x,x) > 2G(x,\Delta)$ for all $x < \Delta$.
    \item $G(i,x) > G(x,\Delta)$ for all $1 \le i < x < \Delta$.
    \item $G(x,\Delta) > 0$ for all $2 \le x \le \Delta - 1$.
\end{enumerate}
\end{lemma}

\begin{proof}
(i) By (P2) at $x=1$, $f(1,\Delta) > f(\Delta,\Delta)$. Thus, $C_\Delta > 0$.

(ii) The partial derivative of $G(x,y)$ is:
\begin{equation*}
    \frac{\partial G(x,y)}{\partial y} = -\frac{\partial f(x,y)}{\partial y} - \frac{C_\Delta}{y^2}.
\end{equation*}
By (P3) and $C_\Delta > 0$, $\frac{\partial G}{\partial y} < 0$. Therefore, $G(x,j) > G(x,\Delta)$ for $j < \Delta$.

(iii) Evaluating $G(x,x)$ and $2G(x,\Delta)$ gives:
\begin{align*}
    G(x,x) &= f(\Delta,\Delta) - f(x,x) + 2C_\Delta \left( \frac{1}{x} - \frac{1}{\Delta} \right), \\
    2G(x,\Delta) &= 2f(\Delta,\Delta) - 2f(x,\Delta) + 2C_\Delta \left( \frac{1}{x} - \frac{1}{\Delta} \right).
\end{align*}
Subtracting yields:
\begin{equation*}
    G(x,x) - 2G(x,\Delta) = -f(x,x) + 2f(x,\Delta) - f(\Delta,\Delta).
\end{equation*}
By (P1), $f(x,x) = f(\Delta,\Delta)$, simplifying the expression to $2\big(f(x,\Delta) - f(\Delta,\Delta)\big)$. By (P2), this is positive, establishing $G(x,x) > 2G(x,\Delta)$.

(iv) Expanding the inequality $G(i,x) > G(x,\Delta)$ using the definition of $G$ yields:
\begin{equation*}
    f(\Delta,\Delta) - f(i,x) + C_\Delta \left( \frac{1}{i} + \frac{1}{x} - \frac{2}{\Delta} \right) > f(\Delta,\Delta) - f(x,\Delta) + C_\Delta \left( \frac{1}{x} - \frac{1}{\Delta} \right).
\end{equation*}
Canceling common terms and applying symmetry $f(x,\Delta) = f(\Delta,x)$, this simplifies to:
\begin{equation*}
    \frac{C_\Delta}{i} - f(i,x) > \frac{C_\Delta}{\Delta} - f(\Delta,x).
\end{equation*}
Defining $\Phi(t) = \frac{C_\Delta}{t} - f(t,x)$, the inequality is exactly equivalent to $\Phi(i) > \Phi(\Delta)$. The derivative is:
\begin{equation*}
    \Phi'(t) = -\frac{C_\Delta}{t^2} - \frac{\partial f(t,x)}{\partial t}.
\end{equation*}
For $t < x$, (P4) ensures $\Phi'(t) < 0$. For $t > x$, (P3) implies $\frac{\partial f(t,x)}{\partial t} > 0$, making both terms negative, so $\Phi'(t) < 0$ trivially. Since $\Phi(t)$ is strictly decreasing on $[1, \Delta]$ and $i < \Delta$, the inequality $\Phi(i) > \Phi(\Delta)$ holds.

(v) $G(x,\Delta) > 0$ follows directly from (P5) and the definition of $G(x,\Delta)$ in \eqref{eq:algebraic_G}.
\end{proof}

With the penalty bounds secured by Lemma~\ref{lem:universal_bounds}, we now establish our main structural characterization of extremal trees relative to the basis $\mathcal{B} = \{(1,\Delta), (\Delta,\Delta)\}$.

\begin{theorem}[Extremal Tree Degree Structure]
\label{thm:alternative_structure}
Let $T$ be a tree of order $n$ and maximum degree $\Delta \ge 3$. If $f(x,y)$ satisfies properties (P1)–(P5), then any tree $T$ maximizing $TI(T)$ has the following edge-multiplicity structure:
\begin{itemize}
    \item If $n \equiv 2 \pmod{\Delta-1}$, $T$ contains only edges of types $(1,\Delta)$ and $(\Delta,\Delta)$.
    \item Otherwise, $T$ contains a single internal vertex of degree $x$, yielding $m_{x,\Delta} = x$ alongside the basic edges $(1,\Delta)$ and $(\Delta,\Delta)$, where $x \in \{2, \dots, \Delta - 1\}$ is the unique integer satisfying $n \equiv x+1 \pmod{\Delta-1}$.
\end{itemize}
\end{theorem}

\begin{proof}
Let $n_k$ be the number of vertices of degree $k$. The vertex and degree sums are:
\begin{align}
    \sum_{k=1}^{\Delta} n_k &= n, \label{eq:alt_sum_v} \\
    \sum_{k=1}^{\Delta} k n_k &= 2n - 2. \label{eq:alt_sum_d}
\end{align}
Subtracting \eqref{eq:alt_sum_v} from \eqref{eq:alt_sum_d} isolates $n_\Delta$:
\begin{equation}
    (\Delta-1)n_\Delta = n - 2 - \sum_{x=2}^{\Delta-1} (x-1)n_x.
\end{equation}
Taking this equation modulo $\Delta - 1$ yields the constraint for internal vertices:
\begin{equation}
    \label{eq:alt_modulo}
    \sum_{x=2}^{\Delta-1} (x-1)n_x \equiv n - 2 \pmod{\Delta - 1}.
\end{equation}

To maximize the index, we minimize the total penalty $\mathcal{P} = \sum_{(i,j) \notin \mathcal{B}} G(i,j) m_{i,j}$. For any internal degree $x \in \{2, \dots, \Delta-1\}$, vertex incidence requires:
\begin{equation}
    \sum_{i < x} m_{i,x} + 2m_{x,x} + \sum_{j > x} m_{x,j} = x n_x.
\end{equation}
Multiplying this incidence equation by the boundary gap $G(x, \Delta)$ gives:
\begin{equation}
    x n_x G(x, \Delta) = \sum_{i < x} G(x, \Delta) m_{i,x} + 2G(x, \Delta) m_{x,x} + \sum_{j > x} G(x, \Delta) m_{x,j}.
\end{equation}
Let $\mathcal{P}_x$ be the total penalty contribution of all non-basic edges incident to $x$. By Lemma \ref{lem:universal_bounds}(ii), (iii), and (iv), the strict inequalities $G(i,x) > G(x,\Delta)$, $G(x,x) > 2G(x,\Delta)$, and $G(x,j) > G(x,\Delta)$ guarantee that:
\begin{align}
    \mathcal{P}_x &= \sum_{i < x} G(i,x) m_{i,x} + G(x,x) m_{x,x} + \sum_{j > x} G(x,j) m_{x,j} \nonumber \\
    &\ge \sum_{i < x} G(x,\Delta) m_{i,x} + 2G(x,\Delta) m_{x,x} + \sum_{j > x} G(x,\Delta) m_{x,j} \nonumber \\
    &= x n_x G(x, \Delta). \label{eq:penalty_bound}
\end{align}
Because the bounding inequalities in the Lemma are strict, equality in \eqref{eq:penalty_bound} is attained if and only if $m_{x,\Delta} = x n_x$ and all other non-basic incident edges are exactly zero. Summing over all internal degrees establishes the global lower bound $\mathcal{P} \ge \sum x n_x G(x,\Delta)$.

\textbf{Case 1:} If $n \equiv 2 \pmod{\Delta - 1}$, the minimal non-negative solution to the congruence \eqref{eq:alt_modulo} is $n_x = 0$ for all $2 \le x \le \Delta - 1$. Since $G(x,\Delta) > 0$ by Lemma \ref{lem:universal_bounds}(v), setting $n_x = 0$ achieves the absolute minimum $\mathcal{P} = 0$. By the strict equality condition, this restricts the tree entirely to the basis edges $m_{1,\Delta}$ and $m_{\Delta,\Delta}$. 

\textbf{Case 2:} If $n \not\equiv 2 \pmod{\Delta - 1}$, let $R \in \{1, \dots, \Delta-2\}$ be the unique remainder such that $R \equiv n - 2 \pmod{\Delta - 1}$. By \eqref{eq:alt_modulo}, the sum $S = \sum_{x=2}^{\Delta-1} (x-1)n_x$ satisfies $S = k(\Delta - 1) + R$ for some integer $k \ge 0$. Thus, $S \ge R > 0$. 

Since $S > 0$, at least one internal vertex exists, establishing $\sum_{x=2}^{\Delta-1} n_x \ge 1$. 

Because $G(x,\Delta) > 0$, minimizing the penalty lower bound requires minimizing the coefficient sum $\sum x n_x$. Algebraic decomposition of this sum yields:
\begin{equation}
    \sum_{x=2}^{\Delta-1} x n_x = \sum_{x=2}^{\Delta-1} (x-1)n_x + \sum_{x=2}^{\Delta-1} n_x = S + \sum_{x=2}^{\Delta-1} n_x \ge R + 1.
\end{equation}
Equality holds if and only if $S = R$ and $\sum n_x = 1$. This exact condition requires $n_x = 1$ for a single internal degree $x$ (with all other $n_k = 0$), directly implying $x - 1 = R$. 

Substituting $R \equiv n - 2 \pmod{\Delta - 1}$ yields $n \equiv x + 1 \pmod{\Delta - 1}$. The minimal global penalty is exactly $x G(x, \Delta)$, which, by the equality condition of \eqref{eq:penalty_bound}, is achieved if and only if $m_{x,\Delta} = x$.
\end{proof}

\subsection{Applications to Degree-Homogeneous Topological Indices}

We now apply Theorem~\ref{thm:alternative_structure} to the five degree-homogeneous VDB invariants listed in Table~\ref{tab:indices}. We first examine the three indices—$DRSO$, $SDD$, and $HSO$—whose weight functions satisfy conditions \textbf{(P1)}--\textbf{(P5)} universally for all $\Delta \ge 3$.

\begin{corollary}[Extremal Trees for the $DRSO$ Index]
\label{cor:drso_all_delta}
Let $T$ be a tree of order $n$ and maximum degree $\Delta \ge 3$. The trees maximizing the Degree-Ratio Sombor ($DRSO$) index are uniquely characterized by Theorem \ref{thm:alternative_structure} for all $\Delta \ge 3$.
\end{corollary}

\begin{proof}
For the $DRSO$ index, the weight function is $f(x,y) = \sqrt{(x/y)^2 + (y/x)^2}$. 
(P1): $f(x,x) = \sqrt{2} = f(\Delta,\Delta)$. 
(P2): Applying the AM-GM inequality to the numerator yields $x^4+\Delta^4 > 2x^2\Delta^2$ for $x < \Delta$. Thus, $f(x,\Delta) = \frac{\sqrt{x^4+\Delta^4}}{x\Delta} > \sqrt{2}$. 
(P3): For $x < y$, $\frac{\partial f}{\partial y} = \frac{y^4-x^4}{x y^3 \sqrt{x^4+y^4}} > 0$.

The constant $C_\Delta$ evaluates to $C_\Delta = \frac{\sqrt{\Delta^4+1} - \Delta\sqrt{2}}{\Delta-1}$. Since $\sqrt{\Delta^4+1} > \Delta^2$, it satisfies:
\begin{equation*}
    C_\Delta > \frac{\Delta^2 - \Delta\sqrt{2}}{\Delta-1}.
\end{equation*}
The inequality $\frac{\Delta^2 - \Delta\sqrt{2}}{\Delta-1} \ge \Delta - 1$ simplifies to $\Delta(2-\sqrt{2}) \ge 1$, which holds for all $\Delta \ge 3$. Thus, $C_\Delta > \Delta - 1$.

For (P4), evaluating the partial derivative $\left| \frac{\partial f(t,x)}{\partial t} \right| = \frac{x^4-t^4}{t^2 x \sqrt{t^4+x^4}}$ reduces the constraint $\frac{C_\Delta}{t^2} > \left| \frac{\partial f}{\partial t} \right|$ to:
\begin{equation*}
    C_\Delta > \frac{x^4-t^4}{x \sqrt{t^4+x^4}}, \quad 1 \le t < x < \Delta.
\end{equation*}
The right-hand side decreases with respect to $t$. Since $t \ge 1$, its maximum on this interval is attained at $t = 1$:
\begin{equation*}
    \max_{1 \le t < x} \frac{x^4-t^4}{x \sqrt{t^4+x^4}} = \frac{x^4-1}{x \sqrt{x^4+1}} < \frac{x^4}{x \sqrt{x^4}} = x \le \Delta - 1.
\end{equation*}
Combining these inequalities yields $C_\Delta > \Delta - 1 \ge x > \frac{x^4-1}{x\sqrt{x^4+1}}$, proving (P4) universally.

For (P5), isolating $C_\Delta$ establishes the condition:
\begin{equation*}
    C_\Delta > \frac{\sqrt{x^4+\Delta^4} - x\Delta\sqrt{2}}{\Delta - x}, \quad 2 \le x \le \Delta - 1.
\end{equation*}
Let $F(x)$ denote the right-hand side. Evaluating $F(x)$ at $x=1$ yields exactly $C_\Delta$: $F(1) = \frac{\sqrt{\Delta^4+1} - \Delta\sqrt{2}}{\Delta-1} = C_\Delta$. Therefore, (P5) is equivalent to $F(1) > F(x)$ for all $x \ge 2$.

To prove $F(x)$ is decreasing on $[1, \Delta)$, we analyze its derivative. $F'(x)$ carries the same sign as its reduced numerator $N(x) = \Delta^4 + 2\Delta x^3 - x^4 - \Delta^2\sqrt{2(x^4+\Delta^4)}$. Substituting $x = u\Delta$ for $u \in (0,1)$ and isolating the radical requires proving:
\begin{equation*}
    1 + 2u^3 - u^4 < \sqrt{2(1+u^4)}.
\end{equation*}
Squaring both sides and grouping terms reduces this to:
\begin{equation*}
    1 - 4u^3(1-u) - u^6(2-u)^2 > 0.
\end{equation*}
We bound the third term by completing the square: $u(2-u) = 1 - (1-u)^2 < 1$ for all $u \in (0,1)$. Thus, $u^2(2-u)^2 < 1$, bounding $u^6(2-u)^2 < u^4$. Substituting this upper bound reduces the inequality to:
\begin{equation*}
    1 - 4u^3(1-u) - u^4 = 3u^4 - 4u^3 + 1 > 0.
\end{equation*}
Factoring the polynomial yields $(u-1)^2(3u^2 + 2u + 1)$. Because $u \in (0,1)$, $(u-1)^2 > 0$ and $3u^2 + 2u + 1 > 0$, proving $N(x) < 0$. Consequently, $F'(x) < 0$, guaranteeing $F(x)$ is strictly decreasing on $[1, \Delta)$. 

Since $1 < 2 \le x$, it follows that $C_\Delta = F(1) > F(x)$, satisfying (P5). Theorem \ref{thm:alternative_structure} therefore governs the extremal topology for all $\Delta \ge 3$.
\end{proof}

For the Symmetric Division Deg ($SDD$) index, the rational form of the weight function simplifies the verification of conditions \textbf{(P1)}--\textbf{(P5)} considerably.

\begin{corollary}[Extremal Trees for the $SDD$ Index]
\label{cor:sdd_all_delta}
Let $T$ be a tree of order $n$ and maximum degree $\Delta \ge 3$. The trees maximizing the Symmetric Division Deg ($SDD$) index are characterized by Theorem \ref{thm:alternative_structure} for all $\Delta \ge 3$.
\end{corollary}

\begin{proof}
For the $SDD$ index, the weight function is $f(x,y) = \frac{x^2+y^2}{xy} = \frac{x}{y} + \frac{y}{x}$. 
(P1): $f(x,x) = 2 = f(\Delta,\Delta)$. 
(P2): For $x < \Delta$, $(x-\Delta)^2 > 0$ gives $x^2+\Delta^2 > 2x\Delta$, so $f(x,\Delta) = \frac{x^2+\Delta^2}{x\Delta} > 2 = f(\Delta,\Delta)$. 
(P3): For $x < y$, $\frac{\partial f}{\partial y} = \frac{y^2-x^2}{xy^2} > 0$.

The multiplier $C_\Delta$ simplifies to
\begin{equation*}
    C_\Delta = \frac{\Delta}{\Delta - 1} \big( f(1,\Delta) - f(\Delta,\Delta) \big) = \frac{\Delta}{\Delta - 1} \left( \Delta + \frac{1}{\Delta} - 2 \right) = \frac{\Delta^2 - 2\Delta + 1}{\Delta - 1} = \Delta - 1.
\end{equation*}

For (P4), we have $\left| \frac{\partial f(t,x)}{\partial t} \right| = \frac{x^2-t^2}{xt^2}$. The condition $\frac{C_\Delta}{t^2} > \left| \frac{\partial f(t,x)}{\partial t} \right|$ reduces to
\begin{equation*}
    \frac{\Delta - 1}{t^2} > \frac{x^2-t^2}{xt^2} \iff \Delta - 1 > \frac{x^2-t^2}{x} = x - \frac{t^2}{x}, \quad 1 \le t < x < \Delta.
\end{equation*}
Since $x \le \Delta - 1$ and $t \ge 1$, we have $x - \frac{t^2}{x} < x \le \Delta - 1$, so (P4) holds for all $\Delta \ge 3$.

For (P5), substituting $f$ and $C_\Delta$ into $f(\Delta,\Delta) - f(x,\Delta) + C_\Delta \left( \frac{1}{x} - \frac{1}{\Delta} \right)$ gives
\begin{equation*}
    2 - \left( \frac{x}{\Delta} + \frac{\Delta}{x} \right) + (\Delta - 1) \left( \frac{\Delta - x}{x\Delta} \right) = \frac{2x\Delta - x^2 - \Delta^2 + \Delta^2 - x\Delta - \Delta + x}{x\Delta} = \frac{(x-1)(\Delta-x)}{x\Delta}.
\end{equation*}
Since $2 \le x \le \Delta - 1$, both $(x-1)$ and $(\Delta-x)$ are positive, verifying (P5). By Theorem \ref{thm:alternative_structure}, the result holds for all $\Delta \ge 3$.
\end{proof}

A similar universal validity across all $\Delta \ge 3$ holds for the Hyperbolic Sombor ($HSO$) index.

\begin{corollary}[Extremal Trees for the $HSO$ Index]
\label{cor:hso_all_delta}
Let $T$ be a tree of order $n$ and maximum degree $\Delta \ge 3$. The trees maximizing the Hyperbolic Sombor ($HSO$) index are uniquely characterized by Theorem \ref{thm:alternative_structure} for all $\Delta \ge 3$.
\end{corollary}

\begin{proof}
For $x \le y$, the $HSO$ index weight function is $f(x,y) = \frac{\sqrt{x^2+y^2}}{x}$. 
(P1): $f(x,x) = \frac{\sqrt{2x^2}}{x} = \sqrt{2} = f(\Delta,\Delta)$. 
(P2): For $x < \Delta$, $f(x,\Delta) = \sqrt{1 + (\Delta/x)^2} > \sqrt{2}$. 
(P3): For $x < y$, $\frac{\partial f}{\partial y} = \frac{y}{x\sqrt{x^2+y^2}} > 0$.

The structural multiplier is $C_\Delta = \frac{\Delta}{\Delta - 1} \left( \sqrt{1+\Delta^2} - \sqrt{2} \right)$.

For (P4), $\left| \frac{\partial f(t,x)}{\partial t} \right| = \frac{x^2}{t^2\sqrt{t^2+x^2}}$. The condition $\frac{C_\Delta}{t^2} > \left| \frac{\partial f(t,x)}{\partial t} \right|$ reduces to:
\begin{equation*}
    C_\Delta > \frac{x^2}{\sqrt{t^2+x^2}}, \quad 1 \le t < x < \Delta.
\end{equation*}
The right-hand side is maximized at $t=1$, giving the upper bound $\frac{x^2}{\sqrt{1+x^2}} < x \le \Delta - 1$. Since $\sqrt{\Delta^2+1} > \Delta$, we have:
\begin{equation*}
    C_\Delta = \frac{\Delta(\sqrt{\Delta^2+1} - \sqrt{2})}{\Delta-1} > \frac{\Delta^2 - \Delta\sqrt{2}}{\Delta-1}.
\end{equation*}
The inequality $\frac{\Delta^2 - \Delta\sqrt{2}}{\Delta-1} \ge \Delta - 1$ reduces to $\Delta(2-\sqrt{2}) \ge 1$, which holds for $\Delta \ge 3$. Thus, $C_\Delta > \Delta - 1 > \frac{x^2}{\sqrt{1+x^2}}$, satisfying (P4).

For (P5), substituting $f(x,y)$ into $f(\Delta,\Delta) - f(x,\Delta) + C_\Delta \left( \frac{1}{x} - \frac{1}{\Delta} \right) > 0$ and isolating $C_\Delta$ yields:
\begin{equation}\label{eq:hso_bound}
    C_\Delta > \frac{x\Delta}{\Delta - x} \left( \frac{\sqrt{x^2+\Delta^2}}{x} - \sqrt{2} \right) = \frac{\Delta}{\Delta - x} \left( \sqrt{x^2+\Delta^2} - x\sqrt{2} \right).
\end{equation}

Let $R(x, \Delta)$ be the right-hand side of \eqref{eq:hso_bound}. Since $R(1, \Delta) = C_\Delta$, (P5) requires $R(1, \Delta) > R(x, \Delta)$ for $2 \le x \le \Delta - 1$. The derivative with respect to $x$ is:
\begin{equation*}
    \frac{\partial R}{\partial x} = \frac{\Delta}{(\Delta-x)^2} \left[ \left( \frac{x}{\sqrt{x^2+\Delta^2}} - \sqrt{2} \right)(\Delta - x) + \sqrt{x^2+\Delta^2} - x\sqrt{2} \right].
\end{equation*}
Simplifying the expression yields:
\begin{equation*}
    \frac{\partial R}{\partial x} = \frac{\Delta^2}{(\Delta-x)^2} \left[ \frac{x+\Delta}{\sqrt{x^2+\Delta^2}} - \sqrt{2} \right].
\end{equation*}
The condition $\frac{x+\Delta}{\sqrt{x^2+\Delta^2}} < \sqrt{2}$ is equivalent to:
\begin{equation*}
    (x+\Delta)^2 < 2(x^2+\Delta^2) \implies x^2 + 2x\Delta + \Delta^2 < 2x^2 + 2\Delta^2 \implies 0 < (\Delta - x)^2.
\end{equation*}
Since $(\Delta - x)^2 > 0$ for $x < \Delta$, we have $\frac{\partial R}{\partial x} < 0$. Thus, $R(x, \Delta)$ is strictly decreasing on $[1, \Delta)$. 

Therefore, $R(1, \Delta) > R(x, \Delta)$ for all $x \ge 2$, satisfying (P5). With (P1)–(P5) verified for all $\Delta \ge 3$, Theorem \ref{thm:alternative_structure} applies.
\end{proof}

In contrast to the preceding three invariants, the Arithmetic-Geometric ($AG$) and Diminished Sombor ($DSO$) indices exhibit finite degree thresholds beyond which conditions \textbf{(P4)} or \textbf{(P5)} fail and the basis $\mathcal{B} = \{(1,\Delta), (\Delta,\Delta)\}$ ceases to be optimal.

\begin{corollary}[Extremal Trees for the $AG$ Index]
\label{cor:ag_small_delta}
Let $T$ be a tree of order $n$ with maximum degree $3 \le \Delta \le 8$. The trees maximizing the Arithmetic-Geometric ($AG$) index are characterized by Theorem \ref{thm:alternative_structure}.
\end{corollary}

\begin{proof}
For the $AG$ index, the weight function $f(x,y) = \frac{x+y}{2\sqrt{xy}}$ satisfies (P1), (P2), and (P3) for all $\Delta \ge 3$: $f(x,x) = 1 = f(\Delta,\Delta)$, $f(x,\Delta) > 1$ for $x < \Delta$ by the AM-GM inequality, and $\frac{\partial f}{\partial y} = \frac{y-x}{4y\sqrt{xy}} > 0$ for $x < y$.

The multiplier $C_\Delta$ simplifies to $C_\Delta = \frac{\Delta}{\Delta - 1} \left( \frac{1+\Delta}{2\sqrt{\Delta}} - 1 \right) = \frac{\sqrt{\Delta}(\sqrt{\Delta}-1)}{2(\sqrt{\Delta}+1)}$.

For (P4), we require $C_\Delta > \frac{t(x-t)}{4t\sqrt{tx}} = \frac{\sqrt{t}(x-t)}{4\sqrt{x}}$ for $1 \le t < x < \Delta$. Applying the AM-GM inequality to $2t$, $(x-t)$, and $(x-t)$ gives
\begin{equation}
    2t(x-t)^2 \le \left( \frac{2t + (x-t) + (x-t)}{3} \right)^3 = \frac{8x^3}{27}.
\end{equation}
Taking the square root yields $\sqrt{t}(x-t) \le \frac{2x\sqrt{3x}}{9}$, with equality at $t = x/3$. Since $x \le \Delta - 1$, (P4) holds whenever
\begin{equation}
    C_\Delta > \frac{\sqrt{3}(\Delta-1)}{18}.
\end{equation}
Substituting $C_\Delta$, dividing by $(\sqrt{\Delta}-1) > 0$, and clearing denominators yields
\begin{equation}
    \Delta - (3\sqrt{3}-2)\sqrt{\Delta} + 1 < 0,
\end{equation}
which holds for integer $\Delta \ge 3$ if and only if $\Delta \le 8$ (as the upper root is $\Delta_{\max} \approx 8.09$).

Property (P5) requires $1 - \frac{x+\Delta}{2\sqrt{x\Delta}} + C_\Delta \left( \frac{\Delta-x}{x\Delta} \right) > 0$, or equivalently,
\begin{equation}
    C_\Delta > \frac{\sqrt{x\Delta}(\sqrt{\Delta}-\sqrt{x})}{2(\sqrt{\Delta}+\sqrt{x})}.
\end{equation}
Since the right-hand side decreases with respect to $x \ge 2$, its maximum on $[2, \Delta-1]$ is attained at $x = 2$:
\begin{equation}
    \frac{\sqrt{\Delta}(\sqrt{\Delta}-1)}{2(\sqrt{\Delta}+1)} > \frac{\sqrt{2\Delta}(\sqrt{\Delta}-\sqrt{2})}{2(\sqrt{\Delta}+\sqrt{2})}.
\end{equation}
Simplifying this inequality gives
\begin{equation}
    (\sqrt{2}-1)\Delta - \sqrt{\Delta} - (2-\sqrt{2}) < 0,
\end{equation}
which likewise holds for integer $\Delta \ge 3$ if and only if $\Delta \le 8$ ($\Delta_{\max} \approx 8.42$). 

Since both (P4) and (P5) hold for $3 \le \Delta \le 8$ and fail for $\Delta \ge 9$, Theorem \ref{thm:alternative_structure} characterizes the extremal trees in this range.
\end{proof}

For the Diminished Sombor ($DSO$) index, the validity domain of the basis $\mathcal{B} = \{(1,\Delta), (\Delta,\Delta)\}$ narrows further to the class of molecular trees ($\Delta \le 4$), as condition \textbf{(P5)} is violated for all $\Delta \ge 5$.

\begin{corollary}[Extremal Trees for the $DSO$ Index]
\label{cor:dso_small_delta}
Let $T$ be a tree of order $n$ with maximum degree $3 \le \Delta \le 4$. The maximum value of the Diminished Sombor ($DSO$) index is uniquely achieved by the extremal configurations dictated by Theorem \ref{thm:alternative_structure}.
\end{corollary}

\begin{proof}
For the $DSO$ index, the weight function is $f(x,y) = \frac{\sqrt{x^2+y^2}}{x+y}$. 
(P1): $f(x,x) = \frac{\sqrt{2x^2}}{2x} = \frac{\sqrt{2}}{2} = f(\Delta,\Delta)$. 
(P2): For $x < \Delta$, $(x-\Delta)^2 > 0 \implies 2x^2 + 2\Delta^2 > (x+\Delta)^2$, establishing $f(x,\Delta) = \frac{\sqrt{x^2+\Delta^2}}{x+\Delta} > \frac{\sqrt{2}}{2}$. 
(P3): For $x < y$, $\frac{\partial f}{\partial y} = \frac{x(y-x)}{\sqrt{x^2+y^2}(x+y)^2} > 0$.

The structural multiplier is $C_\Delta = \frac{\Delta}{\Delta - 1} \left( \frac{\sqrt{1+\Delta^2}}{1+\Delta} - \frac{\sqrt{2}}{2} \right)$. 

For (P5), substituting $f(x,y)$ into $f(\Delta,\Delta) - f(x,\Delta) + C_\Delta \left( \frac{1}{x} - \frac{1}{\Delta} \right) > 0$ and isolating $C_\Delta$ yields:
\begin{equation}\label{eq:dso_bound}
    C_\Delta > \frac{x\Delta}{\Delta - x} \left( \frac{\sqrt{x^2+\Delta^2}}{x+\Delta} - \frac{\sqrt{2}}{2} \right).
\end{equation}

Let $R(x, \Delta)$ be the right-hand side of \eqref{eq:dso_bound}. Since $C_\Delta = R(1, \Delta)$ by definition, property (P5) is equivalent to requiring $R(1, \Delta) > R(x, \Delta)$ for all $2 \le x \le \Delta - 1$. Evaluating the condition at $x = 2$, we define the difference function $D(\Delta) = R(1, \Delta) - R(2, \Delta)$:
\begin{equation}\label{eq:dso_diff}
    D(\Delta) = \frac{\Delta}{\Delta - 1} \left( \frac{\sqrt{1+\Delta^2}}{1+\Delta} - \frac{\sqrt{2}}{2} \right) - \frac{2\Delta}{\Delta - 2} \left( \frac{\sqrt{4+\Delta^2}}{2+\Delta} - \frac{\sqrt{2}}{2} \right).
\end{equation}

To show (P5) fails for $\Delta \ge 5$, we separate $D(\Delta)$ into two components, $A(\Delta)$ and $B(\Delta)$:
\begin{equation*}
    D(\Delta) = \underbrace{ \frac{\Delta\sqrt{1+\Delta^2}}{\Delta^2 - 1} - \frac{2\Delta\sqrt{4+\Delta^2}}{\Delta^2 - 4} }_{A(\Delta)} + \underbrace{ \frac{\sqrt{2}\Delta^2}{2(\Delta-1)(\Delta-2)} }_{B(\Delta)}.
\end{equation*}

For $\Delta \ge 5$, bounding $A(\Delta)$ using the algebraic inequalities $\sqrt{\Delta^2+1} < \Delta + \frac{1}{2\Delta}$ and $\sqrt{\Delta^2+4} > \Delta + \frac{2}{\Delta} - \frac{2}{\Delta^3}$ gives:
\begin{align*}
    A(\Delta) &< \frac{\Delta^2 + 1/2}{\Delta^2-1} - \frac{2\Delta^2 + 4 - 4/\Delta^2}{\Delta^2-4} \\
    &= -1 + \frac{3}{2(\Delta^2-1)} - \frac{12\Delta^2 - 4}{\Delta^2(\Delta^2-4)}.
\end{align*}

Let $U(\Delta)$ be this upper envelope for $A(\Delta)$. Then $D(\Delta) < U(\Delta) + B(\Delta)$. Reducing $U(\Delta) + B(\Delta)$ to a common denominator yields a rational function whose numerator is a $6$-th degree polynomial in $\Delta$. Omitting the detailed algebraic expansion due to its prohibitive length, the leading term of this numerator evaluates to $\left(\frac{\sqrt{2}}{2} - 1\right)\Delta^6$. Since the leading coefficient is negative, the polynomial diverges to negative infinity. Root analysis confirms the numerator is strictly negative for all $\Delta \ge 5$. Thus, $D(\Delta) < 0$, which implies $R(1, \Delta) < R(2, \Delta)$ for all $\Delta \ge 5$, establishing a definitive structural violation of (P5).

For the remaining isolated domain $\Delta \in \{3, 4\}$, evaluating $R(1, \Delta) > R(x, \Delta)$ explicitly yields:
\begin{itemize}
    \item For $\Delta = 3$: $R(1,3) \approx 0.125$ and $R(2,3) \approx 0.084$. Since $R(1,3) > R(2,3)$, (P5) holds.
    \item For $\Delta = 4$: $R(1,4) \approx 0.156$. The internal degrees yield $R(2,4) \approx 0.153$ and $R(3,4) \approx 0.086$. Since $R(1,4) > \max\{R(2,4), R(3,4)\}$, (P5) holds.
\end{itemize}

Finally, to verify (P4) for the valid domain $\Delta \in \{3,4\}$, the condition $\frac{C_\Delta}{t^2} > \left| \frac{\partial f(t,x)}{\partial t} \right|$ rearranges to $C_\Delta > H(t,x)$, where $H(t,x) = \frac{t^2 x(x-t)}{\sqrt{t^2+x^2}(t+x)^2}$. Since the internal degree $x$ strictly takes integer values, we apply differential analysis to $H(t,x)$ with respect to $t \in [1, x)$ to establish its exact supremum:
\begin{itemize}
    \item For $\Delta = 3$ ($x=2$): Equating the derivative to zero reveals a unique maximum near $t \approx 1.04$, giving $\sup_{t} H(t,2) < 0.10$. Since $C_3 \approx 0.125 > 0.10$, (P4) is unconditionally satisfied.
    \item For $\Delta = 4$ ($x \in \{2, 3\}$): Differential analysis bounds the supremum for $x=2$ strictly below $0.10$, and for $x=3$ reveals a maximum near $t \approx 1.57$, guaranteeing $\sup_{t} H(t,3) < 0.15$. Since $C_4 \approx 0.156 > 0.15$, (P4) strictly holds.
\end{itemize}

Since the required axiomatic properties (P1)–(P5) hold simultaneously if and only if $\Delta \le 4$, Theorem \ref{thm:alternative_structure} characterizes the extremal topology uniquely and exclusively for this domain.
\end{proof}

\subsection{Explicit Maximum Values and Constructive Realization of Extremal Trees}

Having established the validity domains of Theorem~\ref{thm:alternative_structure} for specific invariants, we now derive the explicit maximum value of $TI(T)$ and provide a unified constructive characterization of the extremal tree family $\mathcal{T}_{\mathrm{opt}}(n, \Delta)$. From Theorem~\ref{thm:alternative_structure}, the exact edge multiplicities $m_{1,\Delta}$ and $m_{\Delta,\Delta}$ follow directly from \eqref{eq:m_1_delta_final} and \eqref{eq:m_delta_delta_final}, which in turn determine the vertex degree counts $n_1$, $n_x$, and $n_\Delta$:

\begin{itemize}
    \item \textbf{Case 1:} $n \equiv 2 \pmod{\Delta - 1}$. Since $m_{i,j} = 0$ for all $(i,j) \notin \mathcal{B}$, equations \eqref{eq:m_1_delta_final} and \eqref{eq:m_delta_delta_final} immediately reduce to:
    \begin{equation}\label{eq:case1_edges}
        m_{1,\Delta} = \frac{(\Delta - 2)n + 2}{\Delta - 1}, \qquad 
        m_{\Delta,\Delta} = \frac{n - \Delta - 1}{\Delta - 1}.
    \end{equation}
    Because all pendant vertices are incident solely to edges of type $(1,\Delta)$, we have $n_1 = m_{1,\Delta}$ and $n_x = 0$ for all $2 \le x \le \Delta - 1$. Comparing the vertex sum $n = n_1 + n_\Delta$ with the edge sum $n - 1 = m_{1,\Delta} + m_{\Delta,\Delta}$ yields $n_\Delta = m_{\Delta,\Delta} + 1$, giving:
    \begin{equation*}\label{eq:case1_vertices}
        n_1 = \frac{(\Delta - 2)n + 2}{\Delta - 1}, \qquad 
        n_\Delta = m_{\Delta,\Delta} + 1 = \frac{n - 2}{\Delta - 1}.
    \end{equation*}

    \item \textbf{Case 2:} $n \equiv x + 1 \pmod{\Delta - 1}$ for $x \in \{2, \dots, \Delta - 1\}$. Substituting the single non-basic edge multiplicity $m_{x,\Delta} = x$ into \eqref{eq:m_1_delta_final} and \eqref{eq:m_delta_delta_final} gives:
    \begin{align}
        m_{1,\Delta} &= \frac{(\Delta - 2)n + 2}{\Delta - 1} - \frac{\Delta}{\Delta - 1}\left(\frac{1}{x} - \frac{1}{\Delta}\right)x = \frac{(\Delta - 2)n - \Delta + x + 2}{\Delta - 1}, \label{eq:case2_m1delta}\ \\
        m_{\Delta,\Delta} &= \frac{n - \Delta - 1}{\Delta - 1} + \left[\frac{\Delta}{\Delta - 1}\left(\frac{1}{x} - \frac{1}{\Delta}\right) - 1\right]x = \frac{n - 1 - x\Delta}{\Delta - 1}. \label{eq:case2_mdeltadelta}
    \end{align}
    Since all pendant edges are of type $(1,\Delta)$ and all edges incident to vertices of degree $x$ are of type $(x,\Delta)$, vertex-edge incidence gives $n_1 = m_{1,\Delta}$ and $x n_x = m_{x,\Delta} = x \implies n_x = 1$. Comparing the vertex sum $n = n_1 + 1 + n_\Delta$ with the edge sum $n - 1 = m_{1,\Delta} + x + m_{\Delta,\Delta}$ directly yields $n_\Delta = m_{\Delta,\Delta} + x$, from which we obtain:
    \begin{equation*}\label{eq:case2_vertices}
        n_1 = \frac{(\Delta - 2)n - \Delta + x + 2}{\Delta - 1}, \qquad 
        n_x = 1, \qquad 
        n_\Delta = m_{\Delta,\Delta} + x = \frac{n - 1 - x}{\Delta - 1}.
    \end{equation*}
\end{itemize}

Setting $x = \Delta$ naturally subsumes Case~1 into the framework of Case~2: for $x = \Delta$, the condition $n \equiv \Delta + 1 \equiv 2 \pmod{\Delta - 1}$ recovers Case~1, the penalty vanishes ($G(\Delta,\Delta) = 0$), and the total count of vertices of degree $\Delta$ becomes $n_\Delta + 1 = \frac{n - 1 - \Delta}{\Delta - 1} + 1 = \frac{n - 2}{\Delta - 1}$. Consequently, across all congruence classes $n \equiv x + 1 \pmod{\Delta - 1}$ with $x \in \{2, \dots, \Delta\}$, substituting the optimal edge multiplicities into \eqref{eq:ti_expanded} and \eqref{eq:algebraic_G} yields the explicit maximum value of the topological index:
\begin{align}
    \max_{T \in \mathcal{T}(n,\Delta)} TI(T) &= \mathcal{C}(n, \Delta) - x\, G(x, \Delta) \nonumber \\
    &= f(1,\Delta)\frac{(\Delta - 2)n - \Delta + x + 2}{\Delta - 1} + f(\Delta,\Delta)\frac{n - 1 - x\Delta}{\Delta - 1} + x\, f(x, \Delta), \label{eq:max_TI_explicit}
\end{align}
where $\mathcal{T}(n,\Delta)$ denotes the class of all $n$-vertex trees with maximum degree $\Delta$.

Finally, requiring $m_{\Delta,\Delta} \ge 0$ in \eqref{eq:case1_edges} and \eqref{eq:case2_mdeltadelta} imposes the natural order threshold $n \ge \Delta + 1$ in Case~1, and $n \ge x\Delta + 1$ in Case~2. Recognizing the unified parameter structure for $x \in \{2, \dots, \Delta\}$ allows us to formalize the constructive realization of the extremal trees directly via their internal subtrees.

\begin{definition}
\label{def:extremal_trees}
Let $n$ and $\Delta \ge 3$ satisfy the order threshold, and let $x \in \{2, \dots, \Delta\}$ be the unique integer such that $n \equiv x + 1 \pmod{\Delta - 1}$. We define $\mathcal{T}_{\mathrm{opt}}(n, \Delta)$ as the family of all trees constructed from an underlying internal tree $H$ of order $p = \frac{n + \Delta - 2 - x}{\Delta - 1}$ with maximum degree at most $\Delta$, by executing the following procedure:
\begin{enumerate}
    \item Select exactly one vertex $u^* \in V(H)$ such that $d_H(u^*) = x$ if $x < \Delta$, or $1 \le d_H(u^*) \le \Delta$ if $x = \Delta$.
    \item Attach $x - d_H(u^*)$ pendant vertices to $u^*$, and attach $\Delta - d_H(v)$ pendant vertices to each remaining vertex $v \in V(H) \setminus \{u^*\}$.
\end{enumerate}
\end{definition}

By construction, the internal edges of any $T \in \mathcal{T}_{\mathrm{opt}}(n, \Delta)$ are precisely the edges of $H$, totaling $|E(H)| = p - 1$. 

For $x = \Delta$, the procedure assigns $\Delta - d_H(v)$ leaves to every $v \in V(H)$, raising all $p$ internal vertices to degree $\Delta$. This strictly yields $m_{\Delta,\Delta} = |E(H)| = p - 1$ and $m_{x,\Delta} = 0$. The total number of pendant edges evaluates exactly to:
\begin{equation*}
    m_{1,\Delta} = \sum_{v \in V(H)} (\Delta - d_H(v)) = p\Delta - 2|E(H)| = p\Delta - 2(p - 1) = p(\Delta - 2) + 2,
\end{equation*}
which, by substituting $p = n_\Delta = \frac{n - 2}{\Delta - 1}$, perfectly matches the edge distributions of Case 1.

For $x < \Delta$, the selected vertex $u^*$ receives zero leaves, preserving $d_T(u^*) = d_H(u^*) = x$, while the remaining $p - 1$ internal vertices attain degree $\Delta$. Consequently, all $x$ edges incident to $u^*$ exclusively connect to vertices of degree $\Delta$, securing $m_{x,\Delta} = x$. The remaining internal edges connect vertices of degree $\Delta$, yielding $m_{\Delta,\Delta} = |E(H)| - x = p - 1 - x$. The pendant edge count evaluates to:
\begin{align*}
    m_{1,\Delta} &= \sum_{v \in V(H) \setminus \{u^*\}} (\Delta - d_H(v)) = (p - 1)\Delta - \sum_{v \in V(H) \setminus \{u^*\}} d_H(v) \\
    &= (p - 1)\Delta - (2|E(H)| - d_H(u^*)) = (p - 1)\Delta - 2(p - 1) + x = (p - 1)(\Delta - 2) + x.
\end{align*}
Since $p - 1 = n_\Delta = \frac{n - 1 - x}{\Delta - 1}$, this rigorously recovers the parameters of Case 2.

\medskip

Conversely, let $T$ be a tree satisfying the exact parameters of Theorem \ref{thm:alternative_structure}. Deleting all $n_1$ pendant vertices from $T$ yields a connected induced subtree $H$ comprising all internal vertices. The order of $H$ evaluates precisely to the sum of the internal vertex counts. In Case 1 ($n_x = 0$), the order is $|V(H)| = n_\Delta$, while in Case 2 ($n_x = 1$), the order is $|V(H)| = n_\Delta + 1$. This matches the order of the underlying tree prescribed in Definition \ref{def:extremal_trees} perfectly.

For any vertex $v \in V(H)$, the number of pendant vertices incident to it in $T$ is identically $d_T(v) - d_H(v)$. Theorem \ref{thm:alternative_structure} dictates the exact internal degree distribution $d_T(v)$ for the fully constructed tree $T$:
\begin{itemize}
    \item In Case 1, $d_T(v) = \Delta$ for all $v \in V(H)$. Thus, exactly $\Delta - d_H(v)$ pendant vertices must be attached to each $v$.
    \item In Case 2, exactly one vertex $u^* \in V(H)$ has $d_T(u^*) = x$, while all other vertices $v \in V(H) \setminus \{u^*\}$ have $d_T(v) = \Delta$. The condition $m_{x,\Delta} = x$ enforces $m_{1,x} = 0$, meaning $u^*$ receives exactly $0$ pendant vertices, yielding $d_H(u^*) = x$. Every remaining vertex $v$ requires $\Delta - d_H(v)$ pendant vertices.
\end{itemize}
Since $T$ is reconstructed from $H$ via these exact and necessary leaf attachments, any tree satisfying Theorem \ref{thm:alternative_structure} is inherently generated by Definition \ref{def:extremal_trees}, confirming that $\mathcal{T}_{\mathrm{opt}}(n, \Delta)$ is the unique extremal family.

To visually clarify this topological synthesis, Figure \ref{fig:extremal_example} illustrates an extremal tree from $\mathcal{T}_{\mathrm{opt}}(24, 6)$ for $x = 3$, distinctly highlighting the underlying internal subtree $H$ and the resulting edge classes.

\begin{figure}[htbp]
    \centering
    \begin{tikzpicture}[
        internal/.style={circle, fill=black, inner sep=2pt},
        leaf/.style={circle, draw=black, fill=white, inner sep=1.5pt},
        Hedge/.style={ultra thick, black},
        pendantedge/.style={thin, gray},
        scale=1.2
    ]
        \node[internal, label=0:{$u^* (d=3)$}] (u) at (0,0) {};
        \node[internal, label=20:{$v_1 (d=6)$}] (v1) at (-2,1) {};
        \node[internal, label=0:{$v_2 (d=6)$}] (v2) at (2,1) {};
        \node[internal, label={[label distance=4mm, yshift=1.5cm]-75:{$v_3 (d=6)$}}] (v3) at (0,-2) {};
        \node[internal, label=0:{$v_4 (d=6)$}] (v4) at (2,-2) {};

        \draw[Hedge] (u) -- (v1) node[midway, sloped, above] {\small $(3,6)$};
        \draw[Hedge] (u) -- (v2) node[midway, sloped, above] {\small $(3,6)$};
        \draw[Hedge] (u) -- (v3) node[midway, left] {\small $(3,6)$};
        \draw[Hedge] (v3) -- (v4) node[midway, below] {\small $(6,6)$};

        \foreach \angle in {90, 135, 180, 225, 270} {
            \node[leaf] (l1_\angle) at ($(v1) + (\angle:1.2)$) {};
            \draw[pendantedge] (v1) -- (l1_\angle);
        }
        
        \foreach \angle in {60, 105, 150, 240, 285} {
            \node[leaf] (l2_\angle) at ($(v2) + (\angle:1.2)$) {};
            \draw[pendantedge] (v2) -- (l2_\angle);
        }

        \foreach \angle in {120, 165, 210, 255} {
            \node[leaf] (l3_\angle) at ($(v3) + (\angle:1.2)$) {};
            \draw[pendantedge] (v3) -- (l3_\angle);
        }

        \foreach \angle in {45, 90, 135, 270, 315} {
            \node[leaf] (l4_\angle) at ($(v4) + (\angle:1.2)$) {};
            \draw[pendantedge] (v4) -- (l4_\angle);
        }
        
    \end{tikzpicture}
    \caption{An extremal tree $T \in \mathcal{T}_{\mathrm{opt}}(24, 6)$ for $x = 3$. The bold edges and solid black nodes induce the internal connected subtree $H$ of order $p=5$. As governed by the structural theorem, the selected vertex $u^*$ retains degree $3$ and receives no pendant vertices, yielding exactly three $(3,6)$ non-basic edges. The remaining four internal vertices $v_i$ reach the maximum degree $\Delta = 6$ by attaching $\Delta - d_H(v_i)$ leaves (hollow circles connected via thin gray lines), producing a single $(6,6)$ basis edge and $19$ pendant edges.}
    \label{fig:extremal_example}
\end{figure}
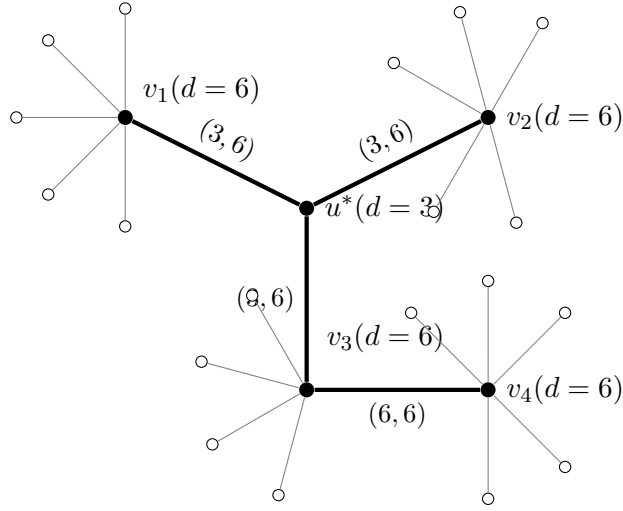

\section{Concluding Remarks and Extension to Graphs of Prescribed Order and Size ($k$-Cyclic Graphs)}

In this paper, we developed a general algebraic procedure that reduces the maximization of degree-homogeneous VDB topological indices over $\Delta$-trees to five analytical conditions \textbf{(P1)}--\textbf{(P5)} on the edge weight function $f(x,y)$. This method provides a complete characterization of the extremal trees across all $\Delta - 1$ congruence classes of $n \pmod{\Delta - 1}$ for the $DRSO$, $SDD$, and $HSO$ indices for all $\Delta \ge 3$, and establishes the exact maximum degree thresholds $\Delta \le 8$ for the $AG$ index and $\Delta \le 4$ for the $DSO$ index. More generally, conditions \textbf{(P1)}--\textbf{(P5)} serve as a direct analytical test applicable to any diagonal-constant VDB invariant.

As noted in the Introduction, the algebraic decomposition in Section~2 extends directly from trees ($k=0$) to the class of connected graphs of prescribed order $n$, size $m = n - 1 + k$ ($k \ge 0$), and maximum degree $\Delta$—that is, connected $k$-cyclic graphs, encompassing unicyclic ($k=1$) and bicyclic ($k=2$) graphs as special cases. For any connected graph $G$ of order $n$ and size $m = n - 1 + k$, the vertex-sum identity \eqref{eq:vertex_sum} remains unchanged, whereas the edge-sum identity \eqref{eq:edge_sum} becomes
\begin{equation*}
    \sum_{1 \le i \le j \le \Delta} m_{i,j} = m = n - 1 + k.
\end{equation*}
Repeating the substitution from Section~2 yields the generalized basic edge multiplicities:
\begin{align*}
    m_{1,\Delta} &= \frac{(\Delta - 2)n - 2k + 2}{\Delta - 1} - \frac{\Delta}{\Delta - 1} \sum_{(i,j) \notin \mathcal{B}} \left( \frac{1}{i} + \frac{1}{j} - \frac{2}{\Delta} \right) m_{i,j}, \\
    m_{\Delta,\Delta} &= \frac{n + (\Delta + 1)(k - 1)}{\Delta - 1} + \sum_{(i,j) \notin \mathcal{B}} \left[ \frac{\Delta}{\Delta - 1} \left( \frac{1}{i} + \frac{1}{j} - \frac{2}{\Delta} \right) - 1 \right] m_{i,j}.
\end{align*}
Because $m$ (or equivalently $k$) appears only in the constant terms, the coefficients multiplying the non-basic edge multiplicities $m_{i,j}$ are independent of the edge count. Consequently, any VDB topological index of a connected $k$-cyclic graph $G$ admits the decomposition
\begin{equation*}
    TI(G) = \mathcal{C}_k(n, \Delta) - \sum_{(i,j) \notin \mathcal{B}} G(i,j)\, m_{i,j},
\end{equation*}
where $\mathcal{C}_k(n, \Delta) = f(1,\Delta)\frac{(\Delta - 2)n - 2k + 2}{\Delta - 1} + f(\Delta,\Delta)\frac{n + (\Delta + 1)(k - 1)}{\Delta - 1}$, and the discrete penalty function $G(x,y)$ is \emph{identical} to that in \eqref{eq:algebraic_G}. 

As a result, Lemma~\ref{lem:universal_bounds} and the validity domains established in Corollaries~\ref{cor:drso_all_delta}--\ref{cor:dso_small_delta} hold without modification for connected graphs with any fixed cyclomatic number $k \ge 0$, while the degree-sum congruence \eqref{eq:alt_modulo} takes the form
\begin{equation*}
    \sum_{x=2}^{\Delta-1} (x-1)n_x \equiv n + 2k - 2 \pmod{\Delta - 1}.
\end{equation*}
Whenever $n$ is sufficiently large relative to $k$ and $\Delta$ to guarantee the graphical realization of a simple connected $k$-cyclic internal subgraph $H$ of order $p = \frac{n + 2k + \Delta - 2 - x}{\Delta - 1}$, the extremal $k$-cyclic graphs maximizing the $DRSO$, $SDD$, $HSO$ ($\Delta \ge 3$), $AG$ ($3 \le \Delta \le 8$), and $DSO$ ($3 \le \Delta \le 4$) indices have the exact same edge-multiplicity structure as in Theorem~\ref{thm:alternative_structure}. Determining the minimal order thresholds $n_0(\Delta, k)$ for simple $k$-cyclic realizations and characterizing the transition to alternative bases when \textbf{(P4)} or \textbf{(P5)} fail (namely, $\Delta \ge 5$ for $DSO$ and $\Delta \ge 9$ for $AG$) remain natural directions for future work.

\end{document}